\documentclass[reqno,a4paper,12pt]{article}
\usepackage{amsmath}
\usepackage{amsthm}
\usepackage{amssymb}
\usepackage{amsfonts}
\usepackage{hyperref}
\usepackage[top=3cm, bottom=3cm, left=3cm, right=3cm]{geometry}
\usepackage{blindtext}
\usepackage[T1]{fontenc}
\usepackage[utf8]{inputenc}
\usepackage{authblk}
\providecommand{\keywords}[1]{\textbf{\textit{Key words:}} #1}
\providecommand{\subjclass}[1]{\textbf{\textit{2000 MSC:}} #1}
\newtheorem{theorem}{Theorem}

\newtheorem{definition}{Definition}

\newtheorem{lemma}{Lemma}

\newtheorem{proposition}{Proposition}
\newtheorem{remark}{Remark}

\title{Asymptotic behaviors of linear advanced systems of differential	equations}
\author{Mouataz Billah Mesmouli}
\affil{\textit{Mathematics Department, Faculty of Science, University of Ha'il, Kingdom of Saudi Arabia} \\
\texttt{mesmoulimouataz@hotmail.com, m.mesmouli@uoh.edu.sa}}

\begin{document}
	\maketitle
	\begin{abstract}
		In this paper, we use the fundamental matrix solution of the system $y^{\prime }\left( t\right) =D\left( t\right) y\left( t\right)$, and the technique of the fixed point theorem to obtain sufficient conditions satisfying the convergence and exponential convergence of
		solutions for the linear system of advanced differential equations. The considered system with multiple variable advanced arguments is discussed as well. The obtained theorems generalize previous results of Dung \cite{dun}, from the one dimension to the $n$ dimension.

\noindent \keywords{Fixed points, asymptotic behaviors, advanced systems, exponential stability, fundamental matrix solution.} \\
    \end{abstract}
\maketitle

\section{Introduction and Preliminaries}

The study of advanced differential equations began in the middle of the last
century by Myschkis \cite{mys} and Bellman \& Cooke \cite{belm}, this type
of equations has been studied considerably by many authors see for example 
\cite{tunc, mal, lie, p1, p2, shah}. However the advanced systems of
differential equations has not been studied before, for this reason, in this
paper we have studied the asymptotic behaviors of linear advanced systems of
differential equations.

The important techniques used in the literature to investigate the
qualitative behaviors of paths of linear and non-linear differential
equations, without finding the explicit solutions, are known as the second
Lyapunov function(al) method, perturbation theory, fixed point method, the
variation of constants formula and so on (see \cite{bur1, bur2} and the
references therein).

Dung N T in \cite{dun}, studied the asymptotic behaviors of the following
linear advanced differential equations%
\begin{equation}
x^{\prime }\left( t\right) +a\left( t\right) x\left( t+h\left( t\right)
\right) +b\left( t\right) x\left( t+r\left( t\right) \right) =0,\ \ t\geq
t_{0},  \label{a}
\end{equation}%
where $a\left( t\right) $ and $b\left( t\right) $ are continuous on $\left[
t_{0},+\infty \right) $, and $h\left( t\right) $, $r\left( t\right) $ are
continuous functions with $h\left( t\right) \geq 0$ and $r\left( t\right)
\geq 0$.

In this paper, we consider the linear system of advanced differential
equations%
\begin{equation}
x^{\prime }\left( t\right) +A\left( t\right) x\left( t+h\left( t\right)
\right) +B\left( t\right) x\left( t+r\left( t\right) \right) =0,\ \ t\geq
t_{0}\geq 0,  \label{eq1.1}
\end{equation}%
in which the functions $h\left( t\right) \geq 0$ and $r\left( t\right) \geq
0 $ are continuous on $\left[ t_{0},+\infty \right) $, $A,B:\left[
t_{0},+\infty \right) \rightarrow 
%TCIMACRO{\U{211d} }%
%BeginExpansion
\mathbb{R}
%EndExpansion
^{2n}$ are $n\times n$ matrices with continuous real-valued functions as its
elements.

Motivated by \cite{dun} and some previous works (\cite{mes1, mes2}), we use
in the analysis the fundamental matrix solution of%
\begin{equation}
y^{\prime }\left( t\right) =D\left( t\right) y\left( t\right) ,  \label{1.4}
\end{equation}%
to invert the system (\ref{eq1.1}) into an integral system which we derive a
fixed point mapping. After then, we define prudently a suitable complete
space, depending on the initial condition, so that the mapping is a
contraction.

We recall now some definitions and results for fundamental matrix, see also 
\cite{c1}.

\begin{definition}
An $n\times n$ matrix function $t\rightarrow \Phi \left( t\right) $, defined
on an open interval $J$, is called a \emph{matrix solution} of the
homogeneous linear system (\ref{1.4}) if each of its columns is a (vector)
solution.
\end{definition}

\begin{definition}
The state transition matrix for the homogeneous linear system (\ref{1.4}) on
the open interval $J$ is the family of fundamental matrix solutions $%
t\rightarrow \Phi \left( t,r\right) $ parametrized by $r\in J$ such that $%
\Phi \left( r,r\right) =I$.
\end{definition}

\begin{proposition}[{\protect\cite[Proposition 2.14]{c1}}]
If $t\rightarrow \Phi \left( t\right) $ is a fundamental matrix solution for
the system (\ref{1.4}) on $J$, then $\Phi \left( t,r\right) :=\Phi \left(
t\right) \Phi ^{-1}\left( r\right) $ is the \emph{state transition matrix}.
Also, the state transition matrix satisfies the Chapman--Kolmogorov
identities%
\begin{equation*}
\Phi (r,r)=I,\ \ \Phi \left( t,s\right) \Phi \left( s,r\right) =\Phi \left(
t,r\right) ,
\end{equation*}%
and the identities%
\begin{equation*}
\Phi \left( t,s\right) ^{-1}=\Phi \left( s,t\right) ,\ \ \frac{\partial \Phi
\left( t,s\right) }{\partial s}=-\Phi \left( t,s\right) A\left( s\right) .
\end{equation*}
\end{proposition}

\begin{remark}
Notice that, $\Phi \left( t,t_{0}\right) $ will be $e^{\left( t-t_{0}\right)
D}$ if $D$ is a constant matrix.
\end{remark}

Throughout this paper, $\Phi \left( t,t_{0}\right) $ will denote a
fundamental matrix solution of the homogeneous (unperturbed) linear problem (%
\ref{1.4}).

\begin{lemma}
\label{lem.a}Let $x\left( t\right) :\left[ t_{0},+\infty \right) \rightarrow 
%TCIMACRO{\U{211d} }%
%BeginExpansion
\mathbb{R}
%EndExpansion
^{n}$ be the solution of (\ref{eq1.1}). Then, the system (\ref{eq1.1}) is
equivalent to%
\begin{eqnarray}
x\left( t\right) &=&\Phi \left( t,t_{0}\right) x\left( t_{0}\right)
+\int_{t_{0}}^{t}\Phi \left( t,s\right) A\left( s\right) \int_{s}^{s+h\left(
s\right) }E_{x}\left( u\right) duds  \notag \\
&&+\int_{t_{0}}^{t}\Phi \left( t,s\right) B\left( s\right)
\int_{s}^{s+r\left( s\right) }E_{x}\left( u\right) duds.  \label{2.1}
\end{eqnarray}%
where $E_{x}\left( u\right) :=A\left( u\right) x\left( u+h\left( u\right)
\right) +B\left( u\right) x\left( u+r\left( u\right) \right) $.
\end{lemma}

\begin{proof}
First we can write%
\begin{equation*}
x\left( t+h\left( t\right) \right) =x\left( t\right) +\int_{t}^{t+h\left(
t\right) }x^{\prime }\left( u\right) du\text{ and }x\left( t+r\left(
t\right) \right) =x\left( t\right) +\int_{t}^{t+r\left( t\right) }x^{\prime
}\left( u\right) du,
\end{equation*}%
substituting these relations into (\ref{eq1.1}), we obtain%
\begin{equation*}
x^{\prime }\left( t\right) +A\left( t\right) x\left( t\right) +A\left(
t\right) \int_{t}^{t+h\left( t\right) }x^{\prime }\left( u\right) du+B\left(
t\right) x\left( t\right) +B\left( t\right) \int_{t}^{t+r\left( t\right)
}x^{\prime }\left( u\right) du=0,
\end{equation*}%
then%
\begin{equation}
x^{\prime }\left( t\right) =-\left( A\left( t\right) +B\left( t\right)
\right) x\left( t\right) -A\left( t\right) \int_{t}^{t+h\left( t\right)
}x^{\prime }\left( u\right) du-B\left( t\right) \int_{t}^{t+r\left( t\right)
}x^{\prime }\left( u\right) du,  \label{eq2.4}
\end{equation}

\noindent Second, we put $D\left( t\right) :=-\left( A\left( t\right) +B\left(
t\right) \right) $ and we put $E_{x}\left( u\right) :=A\left( u\right)
x\left( u+h\left( u\right) \right) +B\left( u\right) x\left( u+r\left(
u\right) \right) $, then, the substitution of (\ref{eq1.1}) in (\ref{eq2.4})
yields%
\begin{equation*}
x^{\prime }\left( t\right) =D\left( t\right) x\left( t\right) +A\left(
t\right) \int_{t}^{t+h\left( t\right) }E_{x}\left( u\right) du+B\left(
t\right) \int_{t}^{t+r\left( t\right) }E_{x}\left( u\right) du.
\end{equation*}%
Now, if we assume that a solution in the interval $\left[ t_{0},\infty
\right) $ is given by%
\begin{equation}
x\left( t\right) =\Phi \left( t,t_{0}\right) \lambda \left( t\right) ,
\label{eq2.5}
\end{equation}%
where $\nu \left( t\right) $ is a differentiable vector valued function to
be determined in the following fashion.

By the product rule for differentiation we have that%
\begin{eqnarray*}
x^{\prime }\left( t\right) &=&\Phi ^{\prime }\left( t,t_{0}\right) \nu
\left( t\right) +\Phi \left( t,t_{0}\right) \nu ^{\prime }\left( t\right) \\
&=&D\left( t\right) \Phi \left( t,t_{0}\right) \nu \left( t\right) +\Phi
\left( t,t_{0}\right) \nu ^{\prime }\left( t\right) .
\end{eqnarray*}%
By the differential equation that $x\left( t\right) $ satisfies on $\left[
t_{0},\infty \right) $, this implies%
\begin{align*}
& D\left( t\right) \Phi \left( t,t_{0}\right) \nu \left( t\right) +\Phi
\left( t,t_{0}\right) \nu ^{\prime }\left( t\right) \\
& =D\left( t\right) x\left( t\right) +A\left( t\right) \int_{t}^{t+h\left(
t\right) }E_{x}\left( u\right) du+B\left( t\right) \int_{t}^{t+r\left(
t\right) }E_{x}\left( u\right) du,
\end{align*}%
then%
\begin{align*}
& D\left( t\right) \Phi \left( t,t_{0}\right) \nu \left( t\right) +\Phi
\left( t,t_{0}\right) \nu ^{\prime }\left( t\right) \\
& =D\left( t\right) \Phi \left( t,t_{0}\right) \nu \left( t\right) \\
& +A\left( t\right) \int_{t}^{t+h\left( t\right) }E_{x}\left( u\right)
du+B\left( t\right) \int_{t}^{t+r\left( t\right) }E_{x}\left( u\right) du.
\end{align*}%
Thus%
\begin{equation*}
\nu ^{\prime }\left( t\right) =\Phi \left( t_{0},t\right) A\left( t\right)
\int_{t}^{t+h\left( t\right) }E_{x}\left( u\right) du+\Phi \left(
t_{0},t\right) B\left( t\right) \int_{t}^{t+r\left( t\right) }E_{x}\left(
u\right) du.
\end{equation*}%
The previous expression, after integrating from $t_{0}$ to $t$ and using the
fact that $\nu \left( t_{0}\right) =x\left( t_{0}\right) $, implies%
\begin{align*}
\nu \left( t\right) & =x\left( t_{0}\right) +\int_{t_{0}}^{t}\Phi \left(
t_{0},s\right) A\left( s\right) \int_{s}^{s+h\left( s\right) }E_{x}\left(
u\right) duds \\
& +\int_{t_{0}}^{t}\Phi \left( t_{0},s\right) B\left( s\right)
\int_{s}^{s+r\left( s\right) }E_{x}\left( u\right) duds,
\end{align*}%
substituting into (\ref{eq2.5}) we get (\ref{2.1}).

The converse implication is easily obtained and the proof is complete.
\end{proof}

\section{ Asymptotic behaviors}

Let $C\left( \left[ t_{0},+\infty \right) ,%
%TCIMACRO{\U{211d} }%
%BeginExpansion
\mathbb{R}
%EndExpansion
^{n}\right) $ is the space of all $n$-vector continuous functions $x\left(
t\right) $ on $\left[ t_{0},+\infty \right) $ such that $x(t_{0})=x_{0}$.

It is seen that $C\left( \left[ t_{0},+\infty \right) ,%
%TCIMACRO{\U{211d} }%
%BeginExpansion
\mathbb{R}
%EndExpansion
^{n}\right) $ is a complete metric space, endowed with the supremum norm 
\begin{equation*}
\Vert x\left( \cdot \right) \Vert =\sup_{t\in \left[ t_{0},+\infty \right)
}|x\left( t\right) |,
\end{equation*}%
where $\left\vert \cdot \right\vert $ denotes the infinity norm for $x\in 
%TCIMACRO{\U{211d} }%
%BeginExpansion
\mathbb{R}
%EndExpansion
^{n}$. Also, if $A$ is an $n\times n$ real matrix, then we define the norm
of $A$ by 
\begin{equation*}
\left\vert A\right\vert =\max_{1\leq i\leq n}\sum_{j=1}^{n}\sup_{t\in \left[
t_{0},+\infty \right) }\left\vert a_{ij}\left( t\right) \right\vert .
\end{equation*}%
Let $\left\{ x^{\ast }\left( t\right) ,t\geq t_{0}\right\} $ be an arbitrary
solution of (\ref{eq1.1}). Then we can define $x_{0}=x^{\ast }\left(
t_{0}\right) $. Thanks to Lemma \ref{lem.a}, we know that $\left\{ x^{\ast
}\left( t\right) ,t\geq t_{0}\right\} $ is a solution of equation (\ref{2.1}%
) with a initial condition $x\left( t_{0}\right) =x_{0}$. So, we define for
all $t\geq t_{0}$ the mapping $\mathcal{H}$ by 
\begin{eqnarray}
\left( \mathcal{H}\varphi \right) \left( t\right) &=&\Phi \left(
t,t_{0}\right) x\left( t_{0}\right) +\int_{t_{0}}^{t}\Phi \left( t,s\right)
A\left( s\right) \int_{s}^{s+h\left( s\right) }E_{x}\left( u\right) duds 
\notag \\
&&+\int_{t_{0}}^{t}\Phi \left( t,s\right) B\left( s\right)
\int_{s}^{s+r\left( s\right) }E_{x}\left( u\right) duds.  \label{H}
\end{eqnarray}%
In this paper we assume that, for, we assume that for all $s_{2}\geq
s_{1}\in \left[ t_{0},\infty \right) $, let us have the uniform bound, in
other words let%
\begin{equation}
\sup_{s_{2}\geq s_{1}\geq t_{0}}\left\Vert \Phi \left( s_{2},s_{1}\right)
\right\Vert \leq K<\infty .  \label{K}
\end{equation}

\begin{theorem}
\label{Thm1}Assume that (\ref{K}) the following conditions hold,%
\begin{equation}
\lim_{t\rightarrow +\infty }\Phi \left( t,t_{0}\right) =0  \label{zero}
\end{equation}%
\begin{equation}
\begin{array}{l}
\int_{t_{0}}^{t}\left\vert \Phi \left( t,s\right) \right\vert \left\vert
A\left( s\right) \right\vert \int_{s}^{s+h\left( s\right) }\left( \left\vert
A\left( u\right) \right\vert +\left\vert B\left( u\right) \right\vert
\right) duds \\ 
+\int_{t_{0}}^{t}\left\vert \Phi \left( t,s\right) \right\vert \left\vert
B\left( s\right) \right\vert \int_{s}^{s+r\left( s\right) }\left( \left\vert
A\left( u\right) \right\vert +\left\vert B\left( u\right) \right\vert
\right) duds:=\alpha <1,%
\end{array}
\label{alfa}
\end{equation}%
Then every solution $\left\{ x\left( t\right) ,t\geq t_{0}\right\} $ of (\ref%
{eq1.1}) with initial condition $x\left( t_{0}\right) $ converges to zero.
\end{theorem}

\begin{proof}
In order to obtain the desired result, it is enough to show that the mapping
(\ref{H}) has an unique solution and this solution converges to zero as $t$
tends to $\infty $. So, we consider a closed subspace $\mathcal{S}$ of $%
C\left( \left[ t_{0},+\infty \right) ,%
%TCIMACRO{\U{211d} }%
%BeginExpansion
\mathbb{R}
%EndExpansion
^{n}\right) $%
\begin{equation*}
\mathcal{S}=\left\{ x\in C\left( \left[ t_{0},+\infty \right) ,%
%TCIMACRO{\U{211d} }%
%BeginExpansion
\mathbb{R}
%EndExpansion
^{n}\right) :\left\Vert x\right\Vert \leq L\text{ and }\lim_{t\rightarrow
\infty }x\left( t\right) =0\right\} .
\end{equation*}

Firstly, we must prove that $\mathcal{H}$ maps $\mathcal{S}$ into itself.

\noindent \textbf{Step 1.} By definition of $\mathcal{S}$, we must show for $%
x\in \mathcal{S}$ that $\left\Vert \mathcal{H}x\right\Vert \leq L$, $t\geq
t_{0}$. We have that, noticing that $\left\Vert x\right\Vert \leq L$ by
definition of $\mathcal{S}$, since (\ref{K}), (\ref{zero}) and (\ref{alfa})
hold, so that%
\begin{align*}
\left\Vert \mathcal{H}x\right\Vert & \leq \sup_{t\in \left[ t_{0},+\infty
\right) }\left\vert \Phi \left( t,t_{0}\right) x\left( t_{0}\right)
\right\vert \\
& +\sup_{t\in \left[ t_{0},+\infty \right) }\left\vert \int_{t_{0}}^{t}\Phi
\left( t,s\right) A\left( s\right) \int_{s}^{s+h\left( s\right) }E_{x}\left(
u\right) duds\right\vert \\
& +\sup_{t\in \left[ t_{0},+\infty \right) }\left\vert \int_{t_{0}}^{t}\Phi
\left( t,s\right) B\left( s\right) \int_{s}^{s+r\left( s\right) }E_{x}\left(
u\right) duds\right\vert \\
& \leq K\left\Vert x_{0}\right\Vert +\alpha L\leq L.
\end{align*}%
we can choose $\left\Vert x_{0}\right\Vert \leq \frac{\left( 1-\alpha
\right) L}{K}$ to obtain $\left\Vert \mathcal{H}x\right\Vert \leq L$ for
every $t\geq t_{0}$.\medskip

\noindent \textbf{Step 2. }We show that for $x\in \mathcal{S},\ \left( 
\mathcal{H}x\right) \left( t\right) \rightarrow 0$ as $t\rightarrow \infty $%
. By definition of $\mathcal{S}$, $x\left( t\right) \rightarrow 0$ as $%
t\rightarrow \infty $. Thus we have%
\begin{align}
\left\vert \left( \mathcal{H}\varphi \right) \left( t\right) \right\vert &
\leq \left\vert \Phi \left( t,t_{0}\right) x\left( t_{0}\right) \right\vert
+\int_{t_{0}}^{t}\left\vert \Phi \left( t,s\right) \right\vert \left\vert
A\left( s\right) \right\vert \int_{s}^{s+h\left( s\right) }\left\vert
E_{x}\left( u\right) \right\vert duds  \notag \\
& +\int_{t_{0}}^{t}\left\vert \Phi \left( t,s\right) \right\vert \left\vert
B\left( s\right) \right\vert \int_{s}^{s+r\left( s\right) }\left\vert
E_{x}\left( u\right) \right\vert duds  \notag \\
& :=I_{1}+I_{2}+I_{3}.  \label{I}
\end{align}%
By (\ref{zero})%
\begin{equation*}
I_{1}=\left\vert \Phi \left( t,t_{0}\right) x\left( t_{0}\right) \right\vert
\rightarrow 0\text{ as }t\rightarrow \infty .
\end{equation*}%
Moreover, it follows from the fact $x\in \mathcal{S}$ that for any $\epsilon
>0$, there exists $T\geq t_{0}$ such that $\left\vert x\left( t\right)
\right\vert <\frac{\epsilon }{2}$ for all $t\geq T$. Hence, we have%
\begin{align*}
I_{2}& =\int_{t_{0}}^{t}\left\vert \Phi \left( t,s\right) \right\vert
\left\vert A\left( s\right) \right\vert \int_{s}^{s+h\left( s\right)
}\left\vert E_{x}\left( u\right) \right\vert duds \\
& =\int_{t_{0}}^{T}\left\vert \Phi \left( t,s\right) \right\vert \left\vert
A\left( s\right) \right\vert \int_{s}^{s+h\left( s\right) }\left\vert
E_{x}\left( u\right) \right\vert duds \\
& +\int_{T}^{t}\left\vert \Phi \left( t,s\right) \right\vert \left\vert
A\left( s\right) \right\vert \int_{s}^{s+h\left( s\right) }\left\vert
E_{x}\left( u\right) \right\vert duds \\
& <\int_{t_{0}}^{T}\left\vert \Phi \left( t,s\right) \right\vert \left\vert
A\left( s\right) \right\vert \int_{s}^{s+h\left( s\right) }\left\vert
E_{x}\left( u\right) \right\vert duds \\
& +\frac{\epsilon }{2}\int_{T}^{t}\left\vert \Phi \left( t,s\right)
\right\vert \left\vert A\left( s\right) \right\vert \int_{s}^{s+h\left(
s\right) }\left( \left\vert A\left( u\right) \right\vert +\left\vert B\left(
u\right) \right\vert \right) duds
\end{align*}%
We observe that $\int_{t_{0}}^{T}\left\vert \Phi \left( t,s\right)
\right\vert \left\vert A\left( s\right) \right\vert \int_{s}^{s+h\left(
s\right) }\left\vert E_{x}\left( u\right) \right\vert duds$ converges to
zero as $t\rightarrow \infty $ due to condition (\ref{zero}). Thus, there
exists $T_{1}\geq T$, such that%
\begin{equation*}
I_{2}<\frac{\epsilon }{2}+\frac{\epsilon }{2}\int_{T}^{t}\left\vert \Phi
\left( t,s\right) \right\vert \left\vert A\left( s\right) \right\vert
\int_{s}^{s+h\left( s\right) }\left( \left\vert A\left( u\right) \right\vert
+\left\vert B\left( u\right) \right\vert \right) duds.
\end{equation*}%
Using (\ref{alfa}) we get $I_{2}<\epsilon $ for all $t\geq T_{1}$. In other
words, we have $I_{2}\rightarrow 0$ as $t\rightarrow \infty $.

Similarly, we also have $I_{3}\rightarrow 0$ as $t\rightarrow \infty $, this
proves that $\left( \mathcal{H}\varphi \right) \left( t\right) \rightarrow 0$
as $t\rightarrow \infty $.

We now prove that $\mathcal{H}$ is a contraction.

\noindent \textbf{Step 3.} Clearly, for each $x\in \mathcal{S}$, we have
that $\mathcal{H}x$ is continuous. Let $x,y\in \mathcal{S}$. For $t\geq
t_{0} $ we get by the condition (\ref{alfa}), we get%
\begin{align*}
& \left\vert \left( \mathcal{H}x\right) \left( t\right) -\left( \mathcal{H}%
y\right) \left( t\right) \right\vert \\
& \leq \int_{t_{0}}^{t}\left\vert \Phi \left( t,s\right) \right\vert
\left\vert A\left( s\right) \right\vert \int_{s}^{s+h\left( s\right)
}\left\vert E_{x}\left( u\right) -E_{y}\left( u\right) \right\vert duds \\
& +\int_{t_{0}}^{t}\left\vert \Phi \left( t,s\right) \right\vert \left\vert
B\left( s\right) \right\vert \int_{s}^{s+r\left( s\right) }\left\vert
E_{x}\left( u\right) -E_{y}\left( u\right) \right\vert duds \\
& \leq \int_{t_{0}}^{t}\left\vert \Phi \left( t,s\right) \right\vert
\left\vert A\left( s\right) \right\vert \int_{s}^{s+h\left( s\right) }\left(
\left\vert A\left( u\right) \right\vert +\left\vert B\left( u\right)
\right\vert \right) duds\left\Vert x-y\right\Vert \\
& +\int_{t_{0}}^{t}\left\vert \Phi \left( t,s\right) \right\vert \left\vert
B\left( s\right) \right\vert \int_{s}^{s+r\left( s\right) }\left( \left\vert
A\left( u\right) \right\vert +\left\vert B\left( u\right) \right\vert
\right) duds\left\Vert x-y\right\Vert \\
& \leq \alpha \left\Vert x-y\right\Vert .
\end{align*}%
Since $\alpha <1$. Thus $\mathcal{H}$ is a contraction on $\mathcal{S}$.
This implies that there is a unique solution to (\ref{eq1.1}) with initial
condition $x_{0}$.
\end{proof}

Let us now recall a fundamental concept (see, for instance, \cite{berez})
that will be used in the next theorem.

\begin{definition}
The differential equation $x\left( t\right) +D\left( t\right) x\left(
t\right) =0$, $t\geq t_{0}$ is called exponentially stable, if there exist $%
M_{0}>0$, $\lambda _{0}>0$ such that,%
\begin{equation}
\left\vert \Phi \left( t,s\right) \right\vert \leq M_{0}e^{-\lambda
_{0}\left( t-s\right) },t_{0}\leq s\leq t<\infty ,  \label{Exp}
\end{equation}%
where $M_{0}$ and $\lambda _{0}$ do not depend on $s$.
\end{definition}

\begin{theorem}
\label{Thm3}Assume that $A(t)$ and $B(t)$ are bounded on $\left[
t_{0},+\infty \right) $ and there exist $M_{0}>0$, $\lambda _{0}>0$ such
that (\ref{Exp}) holds.Then any solution u of (\ref{eq1.1}) is defined for $%
t\geq t_{0}$ and satisfies%
\begin{equation}
\left\vert x\left( t\right) \right\vert \leq Me^{-\lambda t},t_{0}\leq
t<\infty .
\end{equation}
\end{theorem}

\begin{proof}
Let us define another closed subspace of $C\left( \left[ t_{0},+\infty
\right) ,%
%TCIMACRO{\U{211d} }%
%BeginExpansion
\mathbb{R}
%EndExpansion
^{n}\right) $ as%
\begin{equation*}
\mathcal{E}=\left\{ x\in C\left( \left[ t_{0},+\infty \right) ,%
%TCIMACRO{\U{211d} }%
%BeginExpansion
\mathbb{R}
%EndExpansion
^{n}\right) :\exists M,\lambda >0\text{ such that }\left\vert x\left(
t\right) \right\vert \leq Me^{-\lambda t}\ \ \forall t\geq t_{0}\right\} .
\end{equation*}

We will show that $\mathcal{H}\left( \mathcal{E}\right) \subset \mathcal{E}$%
. So, we use the same notation $I_{1}$, $I_{2}$ and $I_{3}$ in (\ref{I}).
Then by (\ref{Exp}), we have%
\begin{equation*}
I_{1}\leq M_{0}\left\Vert x_{0}\right\Vert e^{-\lambda _{0}\left(
t-t_{0}\right) },t\geq t_{0},
\end{equation*}%
without loss of generality, we may assume that $\lambda _{0}\neq \lambda $
for $\lambda $ as in the definition of $\mathcal{E}$.

To estimate $I_{2}$ in (\ref{I}), we observe that $h\left( t\right) ,r\left(
t\right) \geq 0$, and hence%
\begin{eqnarray*}
I_{2} &=&\int_{t_{0}}^{t}\left\vert \Phi \left( t,s\right) \right\vert
\left\vert A\left( s\right) \right\vert \int_{s}^{s+h\left( s\right)
}\left\vert E_{x}\left( u\right) \right\vert duds \\
&\leq &M\int_{t_{0}}^{t}\left\vert \Phi \left( t,s\right) \right\vert
\left\vert A\left( s\right) \right\vert \int_{s}^{s+h\left( s\right) }\left(
\left\vert A\left( u\right) \right\vert +\left\vert B\left( u\right)
\right\vert \right) e^{-\lambda \left( u+h\left( u\right) \right) }duds \\
&\leq &M\overline{A}\left( \overline{A}+\overline{B}\right)
\int_{t_{0}}^{t}\left\vert \Phi \left( t,s\right) \right\vert
\int_{s}^{s+h\left( s\right) }e^{-\lambda u}duds \\
&\leq &\frac{M\overline{A}\left( \overline{A}+\overline{B}\right) }{\lambda }%
\int_{t_{0}}^{t}\left\vert \Phi \left( t,s\right) \right\vert e^{-\lambda
s}\left( 1-e^{-\lambda h\left( s\right) }\right) ds \\
&\leq &\frac{M\overline{A}\left( \overline{A}+\overline{B}\right) M_{0}}{%
\lambda }\int_{t_{0}}^{t}e^{-\lambda _{0}\left( t-s\right) }e^{-\lambda s}ds,
\end{eqnarray*}%
where $\left\vert A\left( t\right) \right\vert \leq \overline{A}$ and $%
\left\vert B\left( t\right) \right\vert \leq \overline{B}$. If $\lambda
<\lambda _{0}$, then we have%
\begin{equation*}
\int_{t_{0}}^{t}e^{-\lambda _{0}\left( t-s\right) }e^{-\lambda
s}ds=e^{-\lambda t}\int_{t_{0}}^{t}e^{-\left( \lambda _{0}-\lambda \right)
\left( t-s\right) }e^{-\lambda s}ds\leq \frac{e^{-\lambda t}}{\lambda
_{0}-\lambda },
\end{equation*}%
and if $\lambda >\lambda _{0}$, then we can write%
\begin{equation*}
\int_{t_{0}}^{t}e^{-\lambda _{0}\left( t-s\right) }e^{-\lambda
s}ds=e^{-\lambda _{0}t}\int_{t_{0}}^{t}e^{-\left( \lambda -\lambda
_{0}\right) \left( t-s\right) }e^{-\lambda s}ds\leq \frac{e^{-\left( \lambda
-\lambda _{0}\right) t_{0}}e^{-\lambda _{0}t}}{\lambda _{0}-\lambda }.
\end{equation*}

From this fact for $I_{2}$, we can write%
\begin{equation*}
I_{2}\leq \left\{ 
\begin{array}{l}
\frac{M\overline{A}\left( \overline{A}+\overline{B}\right) M_{0}}{\lambda
\left( \lambda _{0}-\lambda \right) }e^{-\lambda t},\ \ \lambda <\lambda
_{0}, \\ 
\frac{M\overline{A}\left( \overline{A}+\overline{B}\right) M_{0}}{\lambda
\left( \lambda _{0}-\lambda \right) }e^{-\left( \lambda -\lambda _{0}\right)
t_{0}}e^{-\lambda _{0}t},\ \ \lambda >\lambda _{0}.%
\end{array}%
\right.
\end{equation*}%
In the same way, for $I_{3}$, we also obtain%
\begin{equation*}
I_{3}\leq \left\{ 
\begin{array}{l}
\frac{M\overline{B}\left( \overline{A}+\overline{B}\right) M_{0}}{\lambda
\left( \lambda _{0}-\lambda \right) }e^{-\lambda t},\ \ \lambda <\lambda
_{0}, \\ 
\frac{M\overline{B}\left( \overline{A}+\overline{B}\right) M_{0}}{\lambda
\left( \lambda _{0}-\lambda \right) }e^{-\left( \lambda -\lambda _{0}\right)
t_{0}}e^{-\lambda _{0}t},\ \ \lambda >\lambda _{0}.%
\end{array}%
\right.
\end{equation*}%
As $\left\vert \mathcal{H}\left( x\left( t\right) \right) \right\vert \leq
I_{1}+I_{2}+I_{3}$, we infer that%
\begin{equation*}
\left\vert \left( \mathcal{H}x\right) \left( t\right) \right\vert =\left\{ 
\begin{array}{l}
M_{0}\left\Vert x_{0}\right\Vert e^{\lambda _{0}t_{0}}e^{-\lambda _{0}t}+%
\frac{M\left( \overline{A}+\overline{B}\right) ^{2}M_{0}}{\lambda \left(
\lambda _{0}-\lambda \right) }e^{-\lambda t},\ \ \lambda <\lambda _{0}, \\ 
M_{0}\left\Vert x_{0}\right\Vert e^{\lambda _{0}t_{0}}e^{-\lambda _{0}t}+%
\frac{M\left( \overline{A}+\overline{B}\right) ^{2}M_{0}}{\lambda \left(
\lambda _{0}-\lambda \right) }e^{-\left( \lambda -\lambda _{0}\right)
t_{0}}e^{-\lambda _{0}t},\ \ \lambda >\lambda _{0},%
\end{array}%
\right.
\end{equation*}%
then $\mathcal{H}\left( \mathcal{E}\right) \subset \mathcal{E}$ will be hold
since%
\begin{equation*}
M_{0}\left\Vert x_{0}\right\Vert e^{\lambda _{0}t_{0}}+\frac{M\left( 
\overline{A}+\overline{B}\right) ^{2}M_{0}}{\lambda \left( \lambda
_{0}-\lambda \right) }\leq M\text{ and }\lambda <\lambda _{0}.
\end{equation*}

The remainder of the proof is similar to that of Theorem \ref{Thm1}. So, we
omit it here.
\end{proof}

\begin{remark}
Note that, if $A,B$ are continuous real-valued functions on $\left[
t_{0},+\infty \right) $ to $%
%TCIMACRO{\U{211d} }%
%BeginExpansion
\mathbb{R}
%EndExpansion
$, then Theorems \ref{Thm1} and \ref{Thm3} become the \cite[Theorems 2.3 and
2.5 respectively]{dun}s.
\end{remark}

\section{General Problem}

Now, the methods in the previous section can be extended to the following
system%
\begin{equation}
x^{\prime }\left( t\right) +\sum_{j=1}^{N}A_{j}\left( t\right) x\left(
t+h_{j}\left( t\right) \right) =0.  \label{eq4.1}
\end{equation}%
With a same way in Lemma \ref{lem.a}, suppose that%
\begin{equation*}
\overline{D}\left( t\right) :=-\sum_{j=1}^{N}A_{j}\left( t\right) ,
\end{equation*}%
then%
\begin{equation}
x\left( t\right) =\overline{\Phi }\left( t,t_{0}\right) x\left( t_{0}\right)
+\int_{t_{0}}^{t}\Phi \left( t,s\right) \sum_{j=1}^{N}A_{j}\left( s\right)
\int_{s}^{s+h_{j}\left( s\right) }\overline{E}_{x}\left( u\right) duds,
\end{equation}%
where $\overline{E}_{x}\left( u\right) :=\sum_{j=1}^{N}A_{j}\left( u\right)
x\left( u+h_{j}\left( u\right) \right) $ and $\overline{\Phi }\left(
t,t_{0}\right) $ is the solution of%
\begin{equation*}
x^{\prime }\left( t\right) =\overline{D}\left( t\right) x\left( t\right) .
\end{equation*}

The proof of the following theorem is similar to that of Theorems \ref{Thm1}
and \ref{Thm3}. Hence, we omit it.

\begin{theorem}
\label{Thm2}Suppose that the following conditions hold,%
\begin{equation}
\sup_{s_{2}\geq s_{1}\geq t_{0}}\left\Vert \overline{\Phi }\left(
s_{2},s_{1}\right) \right\Vert \leq K<\infty,
\end{equation}%
\begin{equation}
\int_{t_{0}}^{t}\left\vert \overline{\Phi }\left( t,s\right) \right\vert
\sum_{j=1}^{N}\left\vert A_{j}\left( s\right) \right\vert
\int_{s}^{s+h_{j}\left( s\right) }\sum_{k=1}^{N}\left\vert A_{k}\left(
u\right) \right\vert duds:=\overline{\alpha }<1,
\end{equation}%
(i) If%
\begin{equation}
\lim_{t\rightarrow +\infty }\Phi \left( t,t_{0}\right) =0.
\end{equation}%
Then every solution $\left\{ x\left( t\right) ,t\geq t_{0}\right\} $ of (\ref%
{eq4.1}) with initial condition $x\left( t_{0}\right) $ converges to zero.

\noindent (ii) If $A_{j}\left( t\right) ,j=\overline{1,N}$ are bounded and the
equation $x\left( t\right) +\overline{D}\left( t\right) x\left( t\right) =0$
is exponentially stable, then any solution $\left\{ x\left( t\right) ,t\geq
t_{0}\right\} $ of (\ref{eq4.1}) with initial condition $x\left(
t_{0}\right) $ exponentially converges to zero.
\end{theorem}

\begin{remark}
Note that, if $A_{j}\left( t\right) ,j=\overline{1,N}$ are continuous
real-valued functions on $\left[ t_{0},+\infty \right) $ to $%
%TCIMACRO{\U{211d} }%
%BeginExpansion
\mathbb{R}
%EndExpansion
$, then Theorem \ref{Thm2} become the \cite[Theorem 3.1]{dun}.
\end{remark}

\end{document}